\documentclass[12pt]{amsart}
\usepackage{amscd,amsfonts,amsmath,amssymb,amstext,amsthm}
\newtheorem{theorem} {Theorem} [section]

\newtheorem{proposition}[theorem]{Proposition}
\newtheorem{corollary} [theorem]{Corollary}
\theoremstyle{definition}

\begin{document}
\title{Normal bundles of cycles in flag domains}
\author{Jaehyun Hong, Alan Huckleberry and Aeryeong Seo}
\begin{abstract}
\noindent
A real semisimple Lie group $G_0$ embedded in its complexification $G$ has only finitely many orbits
in any $G$-flag manifold $Z=G/Q$.  The complex geometry of its open orbits $D$ (flag domains) is studied
from the point of view of compact complex submanifolds $C$ (cycles) which arise as orbits of certain distinguished subgroups.
Normal bundles $E$ of the cycles are analyzed in some detail.  It is shown that $E$ is trivial if and only if $D$ is holomorphically
convex, in fact a product of $C$ and a Hermitian symmetric space,  and otherwise $D$ is pseudoconcave.  
The proofs make use of basic results of Sommese and of Snow which are discussed in some detail.
\end{abstract}

\dedicatory{Dedicated to J. A. Wolf on the occasion of his 80${}^{th}$ 
birthday.\\ 
His fundamental work in the area of this paper has been a guiding light.}
\renewcommand{\subjclassname}{\textup{2010} Mathematics Subject Classification}
\maketitle
\renewcommand{\thefootnote}{}
\footnotetext{}
\maketitle

\maketitle
\section {Background and statement of results}
This note involves an interaction of ideas and methods from Lie theory and complex analysis.
Since one of our goals is to emphasize this interaction, we begin with a somewhat detailed
sketch of relevant background from both subjects. 
\subsection {Background on flag domains}
A \emph{flag manifold} of a (connected) complex semi-simple Lie group $G$ is a projective
algebraic $G$-homogeneous space $Z=G/Q$.  First examples are Grassmannians, classical
flag manifolds and quadrics.  If $G$ is defined by a (non-degenerate) bilinear form, such as
the complex symplectic group, then manifolds of maximally isotropic flags appear, e.g., the
Grassmannian of Lagrangian $n$-planes in $\mathbb{C}^{2n}$ equipped with its standard
symplectic structure.
A real form of $G$ is defined at the Lie algebra level as the fixed point algebra $\mathfrak{g}_0$
of an anti-linear Lie algebra involution $\tau :\mathfrak{g}\to \mathfrak{g}$.  Fixing $G$ as a simply-connected
linear algebraic group, we choose $G_0$ to be the subgroup of $G$ corresponding to the embedding
$\mathfrak{g}_0\hookrightarrow \mathfrak{g}$.\\

The action of $G_0$ on a $G$-flag manifold carries a wealth of geometric information.  The simplest
fact concerning this action is that $G_0$ has only finitely many orbits (\cite{W}, see \cite{FHW} for this and
other background).  In particular, $G_0$ 
has open orbits $D$ in $Z$ which are called \emph{flag domains}.  These are of particular importance
in $G_0$-representation theory and, for example, for geometric questions concerning moduli of algebraic
varieties (see, e.g., \cite{GGK}). All results in this note are stated under the assumption that $G_0$
is simple, i.e., that $\mathfrak{g}_0$ is a simple Lie algebra.  Since we are interested in flag domains, this
is no loss of generality, because $D$ splits as a product of manifolds corresponding to the splitting 
of $\mathfrak{g}_0$ as a sum of simple Lie algebras.\\

Maximal compact subgroups $K_0$ of $G_0$ are important for any study of a $G_0$-action.  In particular,
the Riemannian symmetric space associated to $G_0$ is the homogeneous space $M=G_0/K_0$. 
Three basic facts are: $1.$ $K_0$ is actually a maximal subgroup of $G_0$.  $2.$ Any two choices of $K_0$
are conjugate by an element of $G_0$.  $3.$ $G_0/K_0$ is a cell ($K_0$ is a strong deformation retract of
$G_0$).\\

At the level of flag domains a first basic result (proved by Wolf in his fundamental work \cite{W}) is that
there is exactly one $K_0$-orbit in $D$ (called the $K_0$-base cycle and denoted by $C_0$) which is 
a complex manifold.  It can also be characterized as the unique $K_0$-orbit of minimal dimension, and,
from properties of the Mostow-fibration, it follows that $D$ is a real vector bundle over $C_0$.  Of course
this is true for any cycle $C$ which is defined by a choice of $K_0$.  Note that all such choices are
parameterized by the symmetric space $M$.\\  

Now, given a choice of $K_0$, let $K$ denote the connected complex Lie group contained in $G$ which corresponds
to the embedding of $\mathfrak{k}_0+i\mathfrak{k}_0$ in $\mathfrak{g}$.  It is an algebraic (reductive) subgroup of
$G$ which, since $C_0$ is a complex manifold, stabilizes $C_0$.  If $g\in G$, then $C=g(C_0)$ is a flag manifold
of $gKg^{-1}$ which is embedded in $Z$, but for \emph{large} $g$ may not be in $D$.  The set of all cycles $C$ which
arise in this way and which are contained in $D$ is itself a complex manifold, its connected component containing
any base cycle $C_0$ being denoted by $\mathcal{M}_D$.  Numerous of Wolf's works have contributed to our understanding
of $\mathcal{M}_D$ (see \cite{FHW}).

\subsection {Tools from complex analysis}
The goal of this note is to prove several first results dealing with the role of a cycle $C$ in determining complex geometric
properties of $D$.  These are stated below in this introduction and then proved in the sequel. They do not depend on the
choice of the cycle in $\mathcal{M}_D$. Therefore we fix a choice of $K_0$, and its corresponding complexification $K$, thereby
fixing a base cycle $C_0$. Optimally one would hope that the smooth retraction of $D$ to $C_0$ could be realized in a way
that is compatible with the complex analytic properties of $D$.  For example, it would be desirable to have a $K_0$-invariant 
exhaustion $\rho :D\to \mathbb{R}^{\ge 0}$ which takes its minimum along $C_0$ and so that its Levi form $L(\rho):=\frac{i}{2}\partial \bar{\partial}\rho $
has signature properties which lead to vanishing or at least finite dimensionally of holomorphic cohomology.\\

If $q:=dim_{\mathbb{C}}C_0$, then, e.g., $H^q(D,\Omega^q)$ is infinite dimensional, where $\Omega^q$ is the sheaf of holomorphic 
$q$-forms.   If $D$ is measurable, i.e., if it possesses a $G_0$-invariant smooth measure, then Schmidt and Wolf
(\cite{SW}) have shown that a certain naturally defined exhaustion has the desired \emph{Levi convexity} properties so that 
$H^k(D,\mathcal{S})$ vanishes for every coherent sheaf $\mathcal{S}$ and for all $k>q$, this being a consequence of the vanishing theorem
of Andreotti and Grauert (see \cite{AG}).   For applications of classical type in several complex variables, it would be useful to at
least have finite dimensional cohomology for small $k$, e.g, for $k=0,1,2$.  The theorems of Andreotti and Grauert also apply 
in situations where  the exhaustions have \emph{Levi concavity} properties where $L(\rho)$ is required to have (everywhere or at least outside
of a compact subset) certain number of negative eigenvalues.  Thus, at least from the classical viewpoint, it would be desirable to
have a certain degree of \emph{pseudoconcavity}.\\

As a first step we study  the Levi curvature of the normal bundle $E$ of a cycle $C$.  It is known that by using results of Fritzsche (\cite{F}),
as we do below,  pseudoconcave neighborhood bases of the $0$-section of $E$ can be transferred to neighborhood bases of 
$C$ with the same degree of pseudoconcavity.  Following this strategy, and applying Andreotti's finiteness theorem (\cite{A}), if we could
realize this strategy, then we would have the finiteness of the dimension of $H^0(D,\mathcal{S})$ for any coherent sheaf $\mathcal{S}$.
However, this can not hold in general, because $D$ may possess non-constant holomorphic functions.  For example, Hermitian symmetric spaces
of non-compact type are flag domains and they are Stein manifolds!  More generally, if $D$ possesses non-constant holomorphic functions,  then
in a canonical fashion $D$ is a $G_0$-fiber bundle over a Hermitian symmetric space $D_{SS}$ whose fibers are just the cycles in $\mathcal{M}_D$.
In fact this bundle is trivial and consequently the initial flag domain is a product $D=D_{SS}\times C$ (see \cite{H1}).  

\subsection {Statement of results}
Given the above background it is a simple matter to state the results of this note.  For this we first re-emphasize that if $\mathcal{O}(D)\not\cong \mathbb{C}$,
then $D=D_{SS}\times C$ and in particular the normal bundle $E$ of $C$ is trivial.  In fact the converse holds (see Proposition \ref{classification}):
\begin {proposition} Let $C$ be any cycle in a flag domain $D$.  Then
the following are equivalent:
\begin {enumerate}
\item
The normal bundle $E$ of $C$ in $D$ is holomorphically trivial.
\item
The flag domain $D$ is a product $D=D_{SS}\times C$.
\end {enumerate}
\end {proposition}
Our approach involves use of Snow's root theoretic description of the ampleness $a(E)$ which
we outline in some detail below.   This is an invariant of any homogenous vector bundle over
a flag manifold, in our case a cycle, with $a(E)\le dim_{\mathbb{C}}C$.  Equality holds if
$E$ is trivial.  In our context we prove the following (also see Proposition \ref{classification}):
\begin {proposition}
If $C$ is any cycle in a flag domain $D$ and $E$ is its normal bundle, then $a(E)=dim_{\mathbb{C}}C$
if and only if $\mathcal{O}(D)\not\cong\mathbb{C}$ and $E$ is trivial.
\end {proposition}
Thus we must study the situation where $a(E)<dim_{\mathbb{C}}C$.  But this is exactly the setting where
basic results of Sommese (\cite {So1,So2}) apply.  We describe these in some detail below.  The main
point is that the dual bundle $E^*$ has a \emph{norm-like} function whose Levi form has a degree of positivity
which can be estimated by $a(E)$.  Then $E$ has the corresponding negativity and, applying Fritzsche's theorem
as indicated above, we can transfer this to a neighborhood basis of the cycle.  The end result can be stated 
as follows (see Corollary \ref{neighborhoods}):
\begin {proposition}
There is a neighborhood $U$ of $C$ in $D$ and an exhaustion $\rho :U \to \mathbb{R}^{\ge 0}$ with
minimum along $C$ so that outside of $C$ the Levi form $L(\rho)$ has at least $dim_{\mathbb{C}}C-a(E)$
negative eigenvalues.
\end {proposition}
Andreotti's finiteness theorem for $H^0(D,\mathcal{S})$ mentioned above only requires the existence of a relatively compact domain 
$U_c=\{\rho < c\}$ with $L(\rho)$ having at least one negative eigenvalue at each boundary point of $U_c$; in fact much less is required.  In this case we say that 
$D$ is pseudoconcave in the sense of Andreotti.  Thus we have the following remark:
\begin {corollary}
Either $D$ is a product $D=D_{SS}\times C$ of a Hermitian symmetric space and any cycle $C$ contained in $D$ or
$D$ is pseudoconcave in the sense of Andreotti.  In particular, for every coherent sheaf $\mathcal{S}$ the space of sections
$H^0(D,\mathcal{S})$ is finite dimensional.
\end {corollary}
It should be emphasized that, while the  Andreotti-Grauert finiteness/vanishing theorems can be applied to the domain $U$, they
don't directly give results for $D$.

\subsection {Comparison to other results}
In (\cite {H1}) the cycle connectedness equivalence relation was introduced: Two points are equivalent if they are contained in a connected chain of
cycles $D$.  It was shown that $D$ is cycle connected if and only if it is not a product $D=D_{SS}\times C$.  In the cycle connected case the hope
was to fill out a neighborhood $U$ of a given cycle with such chains and to prove the pseudoconcavity of $U$ using supporting cycles at its boundary.
This was done in the $1$-connected case where only chains of length one are needed.  Note that in this case the normal bundle of a cycle would certainly
seem to have a degree of ampleness.  Recently connectedness properties have been further studied in concrete settings (\cite {L, Ha}), and it appears
that $1$-connectedness is not unusual.  In (\cite{K}), in a somewhat more general setting, finite dimensionality of $H^0(D,\mathcal{S})$ was proved, e.g.,
when $D$ is not a product as above. It is very interesting that the germ of the proof can be traced back to Siegel's Schwarz-Lemma method which 
was used by Andreotti to prove finite dimensionality under weak pseudoconcavity assumptions.\\

Most recently in ([HHL]) a method was introduced (completely different from that introduced here) which produces an exhaustion $\rho :D\to \mathbb{R}^{\ge 0}$
where the pseudoconcavity of the sublevel sets $\{\rho <c\}$ is displayed by supporting cycles at the boundaries $\{\rho =c\}$, of course assuming $D$ is
not a product as above.  In important cases the degree of pseudoconcavity can be computed in terms of root-theoretic invariants.
Although $\rho $ is continuous, it is at best piecewise smooth and therefore a direct application of the theorems of Andreotti and Grauert is not possible.\\

Much of the above has been carried out in an attempt to avoid explicit computations involving the normal bundle $E$ of $C$.   Given the approach proposed
here, the only missing ingredient in the general case is the computation of Snow's invariant which will in turn yield the desired formula for $a(E)$ and
the degree of pseudoconcavity of the corresponding neighborhood basis of $C$.   The basic tools for this computation should be available in Part 4 of (\cite{FHW}).
Nevertheless, root theoretic computations of this type can be quite involved.

\section {Spanned and ampleness}
Before going further, let us clarify the several notions that are relevant for our considerations that follow.   First, a holomorphic vector bundle
$E\to X$ is said to be \emph{spanned} if for every $x\in X$ there are global sections $s_1,\ldots ,s_r\in \Gamma (X,E)$ so
that $(s_1(x),\ldots ,s_r(x))$ is a basis of the fiber $E_x$.  If $V=\Gamma(X,E)$ and the \emph{evaluation map}
$\nu :E^*\to V^*$ is defined by $\nu (f_x)(s)=f_x(s(x))$, then $E$ is spanned if and only if $\nu \vert E_x^*$ is a linear
isomorphism for all $x\in X$.  In our context the relevant bundle $E$ is spanned:
\begin {proposition}
The normal bundle $E$ of the cycle $C$ is spanned.
\end {proposition}
\begin {proof}
Since $Z$ is $G$-homogeneous, given
$p\in C$ there exist $\xi_1,\ldots ,\xi_r\in \mathfrak{g}$ with associated holomorphic vector
fields $\hat{\xi}_1,\ldots ,\hat{\xi}_r$ so that the images of $\hat{\xi}_1(p),\ldots , \hat{\xi}_{r}(p)$
in $T_pZ/T_pC$ form a basis of that space.
\end {proof}
The notion of $k$-\emph{ample} is defined as follows (see \cite{So1}).  For this we assume that
$X$ is a compact complex manifold which is always the case in our applications.  Let $\mathbb P(E^*)$ be the quotient of $E^*\setminus X$
by the standard $\mathbb{C}^*$-action defined by $v\mapsto \lambda v$ and let $\zeta^*$ be the tautological bundle
of $1$-dimensional subspaces of $E^*$ regarded as a bundle over $\mathbb P(E^*)$.  Assume that $\zeta^d$ is spanned. If for some such
$d$ the fibers of the evaluation map $\nu_d: (\zeta^d)^*\to \Gamma (\mathbb {P}(E^*),\zeta^d)^*$ are at most $k$-dimensional, then one
says that $E$ is $k$-ample.  The \emph{ampleness} $a(E)$ of $E$ is defined to be the minimum of the $k$ so that $E$ is $k$-ample.\\

If $E$ is spanned, a much simpler definition is possible. This is related to the fact that $E$ is spanned if and only if 
$\pi \times \nu :E^*\to X\times V^*$ is a bundle embedding of $E^*$ as a closed submanifold of the trivial bundle.  Since the
product $X\times V^*$ is holomorphically convex, with Remmert reduction given by its projection onto $V^*$, the Remmert
reduction of $R:E^*\to E^*/{\sim }$ is such that the map $E^*/{\sim} \to Im(\nu)$ defined by the universality property of the Remmert reduction
 is a finite map.  In particular the connected components of the fibers of $\nu $ are just the connected components of the fibers of the Remmert reduction.  Thus, given
$v\in E^*$, the connected component of $\nu $-fiber $\nu^{-1}(\nu (v))$ is simply the maximal,  connected compact analytic set 
containing $v$.\\

On the other hand $E^*$ is naturally identified with  the bundle space $\zeta^*$ with its $0$-section blown down to the $0$-section
$X$ of $E^*$.  Thus, if $\zeta^d$ is spanned, then the above argument applied to $\nu_d$ implies that the connected components
of its fibers are maximal connected compact analytic subsets of the bundle space $(\zeta^d)^*$.  Since $\zeta ^*\to (\zeta^d)^*$ is
a finite map,  the following is immediate.
\begin {proposition}
The ampleness $a(E)$ of a spanned holomorphic vector bundle over a compact complex manifold, is the maximal dimension
of a compact complex subvariety of $E^*$.
\end {proposition}

\begin {corollary}\label{spanned}
The ampleness $a(E)$ of a spanned holomorphic vector bundle $E$ over a compact complex manifold $X$ is the maximum
of the dimensions of the $\nu $-fibers.
\end {corollary}
\begin {proof}
As was shown above, the maximal connected compact analytic subsets in $E^*$ are just the connected components of the
$\nu $-fibers.
\end {proof}

\section {Ampleness of vector bundles}
 We apply basic results of Sommese (\cite{So1,So2}) which we now
 describe for any spanned holomorphic vector bundle $E$ over a compact complex manifold $X$.
\subsection {Levi geometry of bundle spaces}
Sommese's smoothing lemma (Lemma 1.1 of \cite{So2}) 
implies the first basic result (Corollary 1.3 of \cite{So2}):

\begin{proposition} 
\label{modification of metric} 
If $E$ is $k$-ample, then there is a Hermitian metric on $\zeta$ whose curvature form has at most $k$ nonpositive eigenvalues at each point of $\mathbb P(E^*)$.
\end{proposition}  

Given a  Hermitian metric $h_{\zeta}$ on $\zeta$ and the induced metric $h_{\zeta^*}$ on $\zeta^*$, the followings are equivalent.
\begin{enumerate}
\item 
The curvature form $\Theta_{\zeta}(e)$ has  at most $k$ nonpositive  eigenvalues for any $(x,[f])  \in \mathbb {P}(E ^*)$ and for any $e \in \zeta_{E, (x,[f])}$.
 \item
The curvature form $\Theta_{\zeta^*}(\eta)$ has  at most $k$ nonnegative  eigenvalues  for any $(x,[f])  \in \mathbb {P}(E ^*)$ and for any $\eta \in \zeta^*_{E, (x,[f])}$.
\item 
The Levi form of the   norm function of $\zeta^*$ has   at most $k$ nonpositive  eigenvalues.\\
\end{enumerate}

Under the assumption that $E$ is spanned, define a function $\varphi_{E^*}$ on $E^*$ by pulling back the norm function of $\zeta^* $ via the isomorphism $\zeta^*  \backslash Z(\zeta^* ) \simeq E^* \backslash Z(E^*)$, i.e.,
 $\varphi_{E^*} (f ) = h_{\zeta^* , (x, [f])} (f,f) $
 where $x \in X, f \in E^*_x$. Then $\varphi_{E^*}$ satisfies the property
 $$\varphi_{E^*}(\lambda f) = |\lambda|^2\, \varphi_{E^*}(f), \quad \lambda \in \mathbb C, f \in E_x^* $$ 
 and   the Levi form of  $\varphi_{E^*}$   has the same   signature as  the Levi form of the norm function of $\zeta^*$. 
It is important to emphasize that the function $\varphi_{E^*}$ is not the norm function associated to a Hermitian metric on the vector bundle $E^*$ 
but behaves like a norm function. \\

A {\it norm-like}  function on a vector bundle $W$ is   a continuous nonnegative real valued function $\varphi$ on $W$ such that
 \begin{enumerate}
 \item $\varphi(\lambda w) = |\lambda|^2 \varphi( w)$ for any $\lambda \in \mathbb C$ and for any $w \in W$;
 \item $\varphi$ is $C^{\infty}$ on $W \setminus X$. 
\end{enumerate}

With this definition    Proposition \ref{modification of metric} can be stated as follows: 
  
  \begin{proposition} \label{smoothing}
  If $E$ is $k$-ample, then  $E^*$ is $(k+1)$-convex in the sense that there is a {\it norm-like} function on $E^*$ whose Levi form has at most $k$ non-positive eigenvalues at each point of $E^* \setminus X$.
\end{proposition} 

Note that $0$-ample means that the   function $\varphi_{E^*} $ is strictly plurisubharmonic
and therefore the $0$-section $E^*$ can be blown down to a point, i.e., $E$ is ample in
the sense of Grauert (\cite{G}).\\

More generally,  for a holomorphic vector bundle $W$ on a complex manifold $X$, one says that $W$ is $(k+1)$-{\it convex} (respectively, $(k+1)$-{\it concave}) if there exists a norm-like function $\varphi$ on $W$ whose Levi form   has at most $k$ nonpositive (respectively, nonnegative) eigenvalues at each point of $W \setminus X$.  
If, furthermore, the real Hessian of $\varphi$ on $W_x \backslash 0_x$ is everywhere positive definite, then $W$ is said to be {\it strongly} $(k+1)$-{\it convex} 
(respectively, $(k+1)$-{\it concave}). \\

If $E^*$ is equipped with a Hermitian metric, then there is an induced Hermitian metric on $E$.  With these Hermitian  metrics, a direct calculation shows that
the Levi form of the norm function of $E^*$ has at most $k$ non-positive eigenvalues if and only if the Levi form of the norm function of $E$ has at most $k+ r$ non-negative eigenvalues, where $r$ is the rank of $E$.
Sommese's main result (Proposition 1.3 of \cite{So1}) states that the analogous result holds for norm-like function on $E^*$.

\begin{theorem}  \label{Sommese}  
Let $E$ be a holomorphic vector bundle of rank $r$ which is spanned in the sense that
some power of $\zeta $ is spanned, and suppose that $E^*$ is equipped with a
norm-like function.  Then $E$ is equipped with an associated norm-like function
and 
$$
E^*\ \text{is} \  (k+1)\text{-convex} \ \Longrightarrow \ E \ \text{is strongly} \ (k+1+r)\text{-concave}\,.
$$ 
\end{theorem}

In the ampleness language this states the Levi form of the Sommese function on $E$ has
at most $(r+a(E))$ non-negative eigenvalues and, since its restriction  to a fiber $E_x$ is
strictly plurisubharmonic,    this Levi form always has at least $r$ positive eigenvalues.  Thus, as one would
hope, this is set up so that $a(E)=0$, or ampleness in the sense of Grauert, corresponds to the Levi form being maximally
pseudoconcave.\\

\subsection {From the normal bundle to the manifold}
Now let $X$ be a compact submanifold of a complex manifold $Y$ and $E$ be its (holomorphic) normal bundle.  As above
let $r:=rank(E)$.  Here we employ a special case of Fritzsche's Theorem \cite{F} to show
that near $X$ the manifold $Y$ inherits the Levi geometry of $E$.   For this let us say that $Y$ is  $(r+k+1)$-concave
near $X$ if there is an open neighborhood $U$ of $X$ and a smooth function $\rho :U\to \mathbb{R}^{\ge 0}$ with minimum $\{\rho =0\}=X$ and $d\rho \not=0$ otherwise
so that for $\varepsilon >0$ sufficiently small $U_\varepsilon =\{\rho<\varepsilon \}$ is a relatively compact neighborhood of $X$
with $\rho $ being $(r+k+1)$-concave on $U_\varepsilon \setminus X$. In other words, its Levi form of $\rho $ has at most $(r+k)$ non-negative eigenvalues.
In particular, at every point $p$ of every tubular neighborhood $U_\varepsilon $ there is a supporting manifold $F:\Delta \to c\ell(U_\varepsilon)$
where $F$ is a biholomorphic map of a $k$-dimensional polydisk with $F(0)=p$ and $F(\Delta \setminus \{0\})$ contained in $U_\varepsilon$.
As a consequence, it is possible to transfer concavity from the normal bundle space to the manifold:
\begin {theorem}  \cite{F} If the normal bundle $E$ is strongly $(r+k+1)$-concave, then $Y$ is $(r+k+1)$-concave near $X$.
\end {theorem}
Applying Proposition \ref{smoothing} and Theorem \ref{Sommese} yields the following basic fact.
\begin {corollary}
If $E$ is the normal bundle of the compact complex submanifold $X$ in $Y$ and $E$ is spanned, then $Y$ is $(r+a(E)+1)$-concave near $X$.
\end {corollary}
For example, in the case where $\zeta $ is ample in the classical sense, i.e., $a(E)=0$, then this means that the supporting manifolds
have the dimension of the base $X$, i.e., the maximal possible dimension.   In the other extreme where $a(E)=dim_\mathbb{C}X$, the statement is empty.
The following is a more detailed formulation in the case where $a(E)<dim_{\mathbb{C}}X$
\begin {corollary}\label{concavity}
If $a(E)< dim_{\mathbb{C}}X$, then there is an exhaustion function $\rho $ such that (viewed from the side of $X$) the Levi form at every point 
of the boundaries of the sublevel sets $\{\rho < \varepsilon \}$
has at least $dim_{\mathbb{C}}X-a(E)$ negative eigenvalues.
\end {corollary}
In a more general setting where there is a positive, but not determined, number of negative eigenvalues, we simply say that $X$ has a \emph{smoothly 
defined pseudoconcave neighborhood basis} in $Y$.

\section  {Ampleness of homogeneous bundles}
In the case where $X=C$ is a cycle in a flag domain $D$, we may regard the manifold $Y$ in the above notation to
be the ambient $G$-flag manifold $Z$.  For our understanding of the Levi geometry of $D$, at least near $X$,  it is certainly
of interest to be able to compute the ampleness of the normal bundle $E$ of $C$ in $Z$.   Since $C$ is a $K$-orbit in
$Z$, this bundle is $K$-homogeneous bundle.   Snow ($\S2$ of \cite{Sn1} and Theorem 8.3 of \cite{Sn2}) provides a precise
formula for the ampleness $k$ of any homogeneous bundle over any flag manifold:
$$
a(E)=k=dim_{\mathbb{C}}X-ind(-\wedge_{ext}(E_0))\,,
$$
where $E_0$ is the neutral fiber which is the defining $P$-module of $E$ and $\wedge_{ext}(E_0)$ is a certain combinatorial invariant of this
representation.  Since this representation has been computed in the case where $X=C$ is the base cycle in a flag domain (see Part 4 in \cite{FHW}),
at least in theory one can apply Snow's formula, compute the ampleness $a(E)$ of the normal bundle of $C$ and apply the results sketched above
to obtain neighborhood bases of $C$ with a precise degree of pseudoconcavity.  It is the main purpose of this note to point out this possibility and
to give an exact description of the situation when $a(E)=dim_{\mathbb{C}}C$ is maximal.  For this reason, we explain here the essential geometric
aspects of Snow's Theorem.\\

Consider the general situation of a complex semisimple Lie group $G$ and a homogeneous $G$-vector bundle $E=G\times_{P}E_0$ over a flag manifold
$X=G/P$ which is defined by a representation of $P$ on the neutral fiber $E_0$.  Let $\pi:E\to X$ denote the bundle projection with $x_0$ the base point
in $X$ with $G$-isotropy group $G_{x_0}=:P$ and $E_0=\pi^{-1}(x_0)$ over the base point $x_0\in X$.  Denote by $B$ a Borel subgroup in $P$
and fix a maximal torus $T$ in $B$. In the usual root theoretic setup we regard $B$ as the \emph{negative} Borel whose unipotent radical $U=U^{-}$
has Lie algebra $\mathfrak{u^-}=\oplus \mathfrak{g}_\alpha$ where $\alpha $ runs over the negative roots.  Denote by $V$ the $G$-representation
space of sections of $E$.  Throughout we assume that $E$ is spanned in the sense that the evaluation map $V\to E_x$, $s\mapsto s(x)$, is surjective.
As noted above normal bundles of cycles are always spanned. \\

Since $E$ is spanned, the ampleness $a(E)$ is the
maximum of the dimensions of $F_\eta=\nu^{-1}(\eta)$ as $\eta $ runs over all points in $Im(\nu)$(Corollary \ref{spanned}).   Therefore, semi-continuity of fiber dimension
implies that the set $Im(\nu)_{a(E)}$ of points in $Im(\nu)$ with fiber-dimension $a(E)$ is also closed. Consequently  there exists a
$B$-weight vector $\eta_{-\mu}$ with fiber $F:=\nu^{-1}(\eta_{-\mu})$ of dimension $a(E)$, i.e., $\eta_{-\mu}$ is $U$-fixed and
is of weight $-\mu $ in the sense that $t(\eta)=\chi_{-\mu }(t)\eta_{-\mu} $ for $t\in T$ where
$\chi_{-\mu} $ is the character associated to the weight $-\mu \in \mathfrak{t}^*$ where $\mathfrak{t}$ is the Lie algebra of $T$.\\

Now, every irreducible component of $F$ intersects every fiber of $E^*$ in at most one point.  Thus, since $F$ is $U$-invariant, the restriction
of the bundle projection $\pi $ to such a component is bijective
onto a closed $U$-invariant algebraic set in $X$.  Recall that  $U$ has only finitely many orbits in $X$, each of which is a $B$-orbit and referred to as
a $B$-Schubert cell, we have the following observation.
\begin {proposition}
The bundle projection $\pi: E^*\to X$ maps every irreducible component of  $F$ bijectively and $U$-equivariantly onto a $B$-Schubert cell $S$ in $X$.
\end {proposition}
Now $x_0$ is the unique $B$-fixed point in every Schubert variety $S$, in fact in $X$.  Therefore, every irreducible component
of $F$ contains a unique $U$-fixed point which is the
unique point of its intersection with $E_0^*$.  Since $\eta $ is a $T$-weight vector of weight $-\mu $, the following is immediate.
\begin {proposition}
Given an irreducible component of $F$, its unique point $f_{-\mu} $ of intersection with $E_0^*$  is a $T$-weight vector
in the fixed point space $(E_0)^U$ of weight $-\mu $.
 \end {proposition}
Now let $F^0$ be an irreducible component of $F$ of top dimension so that $a(E)=dim_\mathbb{C}F^0$.
Note that if $\hat{E}^*$ is the pull-back of $E^*$ to $\hat{X}=G/B$, then $a(\hat{E})=a(E)+dim_\mathbb{C} P-dim_\mathbb{C} B$.  Thus, in order to simplify the notation,
we may assume that $P=B$.  In this way if $\omega \in W$ is the unique element in the Weyl group of the torus $T$ with $\omega x_0$ the
unique $T$-fixed point in the open $U$-orbit in $S$, then
$$
a(E)=dim_\mathbb{C}F^0=\ell(\omega)
$$
where $\ell(\omega)$ denotes the length of $\omega $ in $W$.\\

If $z$ is the unique point of the intersection of $F^0$ with $E^*_{\omega(x_0)}$, then for $t\in T$ it follows that
$$
t\nu (z)=t\nu(f_{-\mu})=\nu (t(f_\mu))=\nu(\chi_{-\mu}(t)f_\mu)=\chi_{-\mu}(t)\nu(z)\,.
$$
Thus $\omega^{-1}(z)=:f_{-\omega^{-1}\mu}$ is a $T$-weight vector of weight $-\omega^{-1}\mu $ in $E_0^*$.
In particular, $-\mu \in \Lambda ((E_0^*)^U)\cap \omega \Lambda(E_0^*)$, where the set of $T$-weights in a $T$-vector space
$M$ is denoted by $\Lambda (M)$.  Letting $W_0:=\{\omega \in W: \Lambda ((E_0^*)^U)\cap \omega \Lambda(E_0^*)\not=\emptyset\}$,
Snow's Theorem (\cite{Sn1, Sn2}) can be stated as follows.
\begin {theorem} \cite{Sn1}
$a(E)=max_{\omega\in W_0}\{\ell(\omega)\}+dim_{\mathbb{C}}B-dim_{\mathbb{C}}P$
\end {theorem}
\begin {proof}
As indicated above, it is only necessary to give the proof for $B=P$ where it was shown that
$a(E)=\ell(\omega_{F^0})=dim_\mathbb{C}F$ for an explicitly constructed $\omega_{F^0}\in W_0$.
Now let $\omega $ be an arbitrary element of $W_0$, $-\hat{\mu}$ be
the corresponding weight and $\hat{F}:=\nu^{-1}(\nu(f_{-\hat{\mu}}))$ be the  associated $\nu $-fiber through the
$U$-fixed point $f_{-\hat{\mu}}$.  Let $f_{\omega^{-1}\hat{\mu}}$ be the weight vector the existence of which is
guaranteed by $\omega \in W_0$.  It follows that the weight vectors $f_{-\hat{\mu}}$ and $\omega f_{-\omega^{-1}\hat{\mu}}$ are both
in $\hat{F}$.  Since $\hat{F}$ is bijectively mapped by $\pi $ to a Schubert variety $\hat{S}$ and $\ell(\omega)$
is the dimension of some $U$-orbit in $\hat{S}$, it follows that
$$
\ell(\omega)\le dim_\mathbb{C}\hat {F}\le dim_\mathbb{C}F=\ell(\omega_{F^0})\,.
$$
\end {proof}
\section {Normal bundles of cycles in flag domains}
Our application of Snow's Theorem takes place in the situation where the reductive group at hand is the complexification $K$ of a
maximal compact subgroup $K_0$ of the real form $G_0$ which is acting on a $G$-flag manifold $Z=G/Q$.  The vector bundle
$E$ is the normal bundle of a closed $K$-orbit $X$, i.e., of the base cycle $X=C$ in some flag domain $D$.   The parabolic
group $Q$ is defined to be the $G$-isotropy group at some base point in $X$.  Therefore the normal bundle $E$ is defined by the
$K\cap Q$-representation on the quotient space $E_0=\mathfrak{s}/((\mathfrak{q}+\theta \mathfrak{q})\cap \mathfrak{s})$ where $\theta $ is the (complex linear) involution which defines
the Cartan decomposition $\mathfrak{g}=\mathfrak{k}+\mathfrak{s}$.  Thus as a $K\cap Q$-representation, $E_0^*$ is the restriction of its representation
on $\mathfrak{s}^*$ which is defined by the coadjoint representation of $K$.    If $G_0$ is not of Hermitian
type, then $K$ is semisimple.  If  $G_0$ is of Hermitian type, $K$ is reductive with a $1$-dimensional center.  Note that Snow's theorem is valid in the slightly more
general case where $G$ is reductive.  Thus we may apply it to the $K$-homogeneous conormal bundle $E^*$ over the closed $K$-orbit $X$.  We do so
in the simplest case where $a(E)=dim_{\mathbb{C}}X$.\\

In order to state our result we recall that if a flag domain $D$ in $Z=G/Q$ possesses non-constant holomorphic functions, then there is
a canonically defined quotient $Z=G/Q\to G/P=Z_{red}$, where $P$ is the stabilizer in $G$ of the base cycle $C$, so the induced
quotient $D=G_0/L_0\to G_0/K_0=D_{red}$ is the holomorphic reduction of $D$ onto a Hermitian symmetric space $D_{red}$ of non-compact
type which is embedded into its compact dual $Z_{red}$.  The bundle $D\to D_{red}$ is trivial with $C$ as fiber over the $K_0$-fixed point
in $D_{red}$.  In particular, the normal bundle $E$ of $C$ is trivial.  In this situation we simply say that $D$ fibers over a Hermitian symmetric
space (see \cite{H1} for other equivalent conditions).\\

Applying Snow's Theorem we prove the following classification theorem in the simplest case where $a(E)$ is maximal.
\begin{proposition} \label{classification}
If $C$ is a closed $K$-orbit in $Z$ with normal bundle $E$ in a flag domain $D$, then $a(E)=dim_\mathbb{C}C$ if and only if
$E$ is trivial and $D$  fibers over a Hermitian symmetric domain.
\end{proposition}

\begin{proof}
It is enough to assume that $a(E)=dim_\mathbb{C}C$ and show that $D$ fibers over a Hermitian symmetric space.
 If $\frak g_0$ is not of Hermitian type, then $\frak s$ is an irreducible $K$-module and  $\Lambda_{\max}(E_0) = \{\lambda_{\frak s}\}$, where $\lambda_{\frak s}$  is the highest weight of $\frak s $.
 If $\frak g_0$ is of Hermitian type, then $\frak s = \frak s_+ \oplus  \frak s_{-}$ the direct sum of two irreducible $K$-modules $\frak s_+$ and $\frak s_-$,  and  $\Lambda_{\max}(E_0) \subset\{\lambda_+, \lambda_-\}$, where $\lambda_{\pm}$ is the highest weight of $\frak s_{\pm}$. (If $\frak s_+ \not\subset \frak q$ and $\frak s_- \not \subset \frak q$, then $\Lambda_{\max}(E_0) = \{\lambda_+, \lambda_-\}$. If $\frak s_+ \subset \frak q$, then $\Lambda_{\max}(E_0) =\{\lambda_-\}$. If $\frak s_- \subset \frak q$, then
 $\Lambda_{\max}(E_0) = \{\lambda_+\}$.)\\

 Set up the notation as in our above discussion of Snow's Theorem (where $K=G$ in that context).  Since $a(E)=dim_\mathbb{C}C$, there is a weight vector
$f_{-\lambda}$  in $E_0^*$ of weight   $-\lambda \in  \Lambda ({E_0^*}^U)=-\Lambda_{\max}(E_0)$ such that   $ \omega_0 f_{-\lambda} \in E_0^*$
for $\omega_0$ the longest element of the Weyl group $W$.
If $\frak g_0$ is not of Hermitian type, then  $-\omega_0 \lambda$ is the highest  weight of the irreducible representation $\frak s^*$. Thus we have $E_0^*=\frak s^*$, a contradiction to the fact that  $E_0$ is a nontrivial quotient of $\frak s$ (in other words, $E_0^*$ is a proper submodule of $\frak s^*$).\\

 If $\frak g_0$ is of Hermitian type, then $-\lambda$ is the lowest weight of $\frak s_+^*$ (or $\frak s_-^*$) and  $-\omega_0 \lambda$ is the highest  weight of the irreducible representation $\frak s^*_+$ (or $\frak s_-^*$). Thus $E_0^*$ contains either $\frak s_+^*$ or $\frak s_-^*$.
 From $\frak s_+^* = (\frak s/\frak s_-)^*$ and $\frak s_-^* = (\frak  s /\frak s_+)^*$, it follows that $\frak q \cap \frak s$ is contained either in $\frak s_-$ or in $\frak s_+$.
 Then $\frak q$ is contained   either in  $\frak k + \frak s_{-}$ or in $\frak k + \frak s_{+}$, which implies that  there is a $G$-equivariant map from $Z$ to a compact Hermitian symmetric space $\widehat Z$ whose restriction to $D$ is  a   $G_0$-equivariant holomorphic or antiholomorphic map from $D$ to a Hermitian symmetric space of noncompact type  $\widehat{D}=G_0/K_0 \subset \widehat{Z}$.
\end{proof}

Returning to the question of the existence of a pseudoconcave neighborhood basis of $C$, we now apply Corollary \ref{concavity}.
\begin {corollary}\label {neighborhoods}
If $D$ does not fiber over a Hermitian symmetric space, i.e., $D$ is not a product of a Hermitian symmetric space and $C$, then $C$ has a neighborhood 
$U$ in $D$ with a smooth exhaustion $\rho :U\to \mathbb{R}^{\ge 0}$ with minimum along $C$ so that the Levi form $L(\rho)$ has at least $dim_{\mathbb{C}}C-a(E)$ negative eigenvalues where  $a(E)$ is the ampleness of the normal bundle $E$ to $C$ in $D$.

 \end {corollary}

\bigskip

\noindent
\small{Jaehyun Hong and Aeryeong Seo\\
Korean Institute of Advanced Study\\
85 Hoegiro, Dongdaemun-gu.\\ 
Seoul 02455\\ 
Republic of Korea\\

\noindent
Alan Huckleberry\\
Institut f\"ur Mathematik\\
Ruhr Universit\"at Bochum\\
Universit\"atsstrasse 150\\
44780 Bochum, Germany\\
and\\
Jacobs University Bremen\\
Faculty for Mathematics and Logistics\\
Campus Ring 1\\
28759 Bremen, Germany\\

\noindent
jhhong00@kias.re.kr\\ 
ahuck@gmx.de\\ 
aeryeongseo@kias.re.kr}

\end {document}